\documentclass[a4paper,11pt]{article}

\usepackage[dvipdfmx]{graphicx}

\usepackage{amsmath}
\usepackage{amssymb}
\usepackage{amsthm}
\usepackage{latexsym}
\usepackage{amsmath}
\usepackage{amsthm}
\usepackage{txfonts}
\usepackage{bm}
\usepackage{epic,eepic}
\usepackage{color}

\usepackage{geometry}

\newtheorem{theorem}{Theorem}

\newtheorem{corollary}{Corollary}
\newtheorem{lemma}{Lemma}
\newtheorem{remark}{Remark}[section]

\numberwithin{equation}{section}

\newcommand{\todaye}{\the\year/\the\month/\the\day}
\newcommand{\finbox}{\hspace*{\fill}$\rule{0.2cm}{0.2cm}$}

\newcommand{\RR}{{\mathbb{R}}}
\newcommand{\ZZ}{{\mathbb{Z}}}
\newcommand{\dom}{{\rm dom\,}}

\numberwithin{equation}{section}

\title{A Stronger Multiple Exchange Property for \\
M$\sp{\natural}$-concave Functions 
}

\author{Kazuo Murota
\\
School of Business Administration, Tokyo Metropolitan University,
\\
Tokyo 192-0397, Japan; 
murota@tmu.ac.jp
}

\date{June 28, 2017}

\begin{document}

\maketitle

\begin{abstract}
The multiple exchange property for matroid bases
has recently been generalized for valuated matroids and
M$^{\natural}$-concave set functions.
This paper establishes a stronger form of 
this multiple exchange property 
that imposes a cardinality condition 
on the exchangeable subset.
The stronger form immediately implies
the defining exchange property of M$^{\natural}$-concave set functions,
which was not the case with the recently established 
multiple exchange property without the cardinality condition.
\end{abstract}




\section{Introduction}
\label{SCintro}

The concept of M$^{\natural}$-concave functions in discrete convex analysis 
\cite{Fuj05book,Mdcasiam,Mbonn09,MS99gp} 
 has found applications 
in mathematical economics and game theory;
see \cite[Chapter 11]{Mdcasiam}, \cite{Tam04,Tam09book},
and recent survey papers \cite{Mdcaeco16,ST15jorsj}.
M$^{\natural}$-concavity of a set function $f$
is defined in terms of the exchange property that,
for any subsets $X, Y$ and any element $i \in X \setminus Y$,
at least one of  (i) and (ii) holds, where
(i) $f( X) + f( Y ) \leq f( X \setminus  \{ i \} ) + f( Y\cup \{ i \} )$
or (ii) there exists some $j \in Y \setminus X$ such that
$f( X) + f( Y ) \leq 
 f( (X \setminus  \{ i \}) \cup \{ j \} ) + f( (Y \cup \{ i \}) \setminus \{ j \} )$.

It has been shown recently in \cite{Mmultexc17} that 
an M$^{\natural}$-concave set function $f$
has the multiple exchange property that,
for any subsets $X, Y$ and a subset $I \subseteq X \setminus Y$,
there exists $J \subseteq Y \setminus X$ such that
$f( X) + f( Y )   \leq  f((X \setminus I) \cup J) +f((Y \setminus J) \cup I)$.
This result has an economic significance that
the gross substitutes (GS) condition of Kelso and Crawford \cite{KC82}
is in fact equivalent to 
the strong no complementarities (SNC) condition of Gul and Stacchetti \cite{GS99}.
In the special case of M-concave functions,
this multiple exchange property 
gives a quantitative generalization of a classical results in matroid theory
(\cite{Kun86b}, \cite[Section 39.9a]{Sch03})
that the basis family of a matroid enjoys the multiple exchange property,
which says that, 
for two bases $X$ and $Y$ in a matroid and a subset $I \subseteq X \setminus Y$,
there exists a subset $J \subseteq Y \setminus X$ such that
$(X \setminus I) \cup J$ and $(Y \setminus J) \cup I$ are both bases.

The objective of this paper is to 
establish a stronger form of 
the multiple exchange property 
that imposes a cardinality condition $|J| \leq |I|$ 
on the exchangeable subset $J$.
The stronger form immediately implies
the defining exchange property of M$^{\natural}$-concave set functions,
which is not the case with the multiple exchange property
of \cite{Mmultexc17} without the cardinality condition.
The results are described in Section \ref{SCresult} 
and two alternative proofs are
given in Sections \ref{SCproof1direct} and \ref{SCproof2reduce}.

\section{Results}
\label{SCresult}

Let $N$ be a finite set, say,
$N = \{ 1,2,\ldots, n \}$.
For a function
$f: 2^{N} \to \RR \cup \{ -\infty \}$,
$\dom f$ denotes the effective domain of $f$, i.e.,
$\dom f = \{ X \mid f(X) > -\infty \}$.

A function
$f: 2^{N} \to \RR \cup \{ -\infty \}$ 
with $\dom f \not= \emptyset$
is called {\em M$^{\natural}$-concave}
\cite{Mdcasiam,MS99gp}
if,
for any $X, Y \in \dom f$ 
and $i \in X \setminus Y$,
we have  (i)
$X - i \in \dom f$, $ Y + i \in \dom f$, and
\begin{equation}  \label{mconcav1}
f( X) + f( Y ) \leq f( X - i ) + f( Y + i ),
\end{equation}
or (ii) there exists some $j \in Y \setminus X$ such that
$X - i +j \in \dom f$, $ Y + i -j  \in \dom f$, and
\begin{equation}  \label{mconcav2}
f( X) + f( Y ) \leq 
 f( X - i + j) + f( Y + i -j).
\end{equation}
Here we use short-hand notations
$X - i = X \setminus  \{ i \}$,
$Y + i = Y \cup \{ i \}$,
$X - i + j =(X \setminus  \{ i \}) \cup \{ j \}$,
and
$Y + i - j =(Y \cup \{ i \}) \setminus \{ j \}$.
This property
is referred to as the {\em exchange property}.
The exchange property can be expressed more compactly as:
\begin{quote}
\mbox{\bf (M$^{\natural}$-EXC)} \
For any $X, Y \subseteq N$ 
and $i \in X \setminus Y$, we have
\begin{align}
f( X) + f( Y )   &\leq 
   \max\left[ f( X - i ) + f( Y + i ),  
 \right.
\notag \\
    &    \left.        \phantom{\max_{j \in Y \setminus X}} 
 \max_{j \in Y \setminus X}  \{ f( X - i + j) + f( Y + i -j) \}\right] ,
\label{mnatconcavexc2}
\end{align}
\end{quote}
where
$(-\infty) + a = a + (-\infty) = (-\infty) + (-\infty)  = -\infty$ for $a \in \RR$,
$-\infty \leq -\infty$, and
a maximum taken over an empty set
is defined to be $-\infty$.

The {\em multiple exchange property} means 
the following more general form of (M$^{\natural}$-EXC):
\begin{quote}
\mbox{\bf (M$^{\natural}$-EXC$_{\rm\bf m}$)} \ 
For any $X, Y \subseteq N$ and $I \subseteq X \setminus Y$, we have
\begin{align}
f( X) + f( Y )   \leq 
 \max_{J \subseteq Y \setminus X} 
 \{  f((X \setminus I) \cup J) +f((Y \setminus J) \cup I)  \} .
\label{mnatconcavexcmult}
\end{align}
\end{quote}
Here we may specify any subset $I$,
rather than a single element $i$,
in $X \setminus Y$,
and we can always find an exchangeable subset $J \subseteq Y \setminus X$.
It has recently been shown \cite{Mmultexc17} that 
(M$^{\natural}$-EXC) and (M$^{\natural}$-EXC$_{\rm m}$)
are equivalent. 

\begin{theorem}[\cite{Mmultexc17}]  \label{THmnatiffmultexch}
A function
$f: 2^{N} \to \RR \cup \{ -\infty \}$ 
is M$^{\natural}$-concave 
if and only if it has the multiple exchange property 
{\rm (M$^{\natural}$-EXC$_{\rm m}$)}.
\end{theorem}
The content of this theorem lies in the implication 
``(M$^{\natural}$-EXC)
$\Rightarrow$ 
(M$^{\natural}$-EXC$_{\rm m}$).''
It is emphasized, however, that 
``(M$^{\natural}$-EXC$_{\rm m}$)
$\Rightarrow$ 
(M$^{\natural}$-EXC)''
is not obvious and a separate proof is needed also for this direction, 
 though the proof 
\cite[Section 5.2]{Mmultexc17}
is straightforward.

In this paper we are interested in a stronger form of
the multiple exchange property,
in which an additional condition 
$|J| \leq |I|$ is imposed on the exchangeable subset $J$:

\begin{quote}
\mbox{\bf (M$^{\natural}$-EXC$_{\rm\bf ms}$)}
For any $X, Y \subseteq N$ and $I \subseteq X \setminus Y$,
we have
\begin{align}
f( X) + f( Y )   \leq 
 \max_{J \subseteq Y \setminus X, \  |J| \leq |I|} 
 \{  f((X \setminus I) \cup J) +f((Y \setminus J) \cup I)  \}.
\label{mnatconcavexcmult-str}
\end{align}
\end{quote}

The following theorem, the main result of this paper,
states that (M$^{\natural}$-EXC) implies
has the stronger multiple exchange property 
(M$^{\natural}$-EXC$_{\rm ms}$)
with cardinality requirement.

\begin{theorem} \label{THmultexchmnat-str}
Every M$^{\natural}$-concave function
$f: 2^{N} \to \RR \cup \{ -\infty \}$ 
has the stronger multiple exchange property 
{\rm (M$^{\natural}$-EXC$_{\rm ms}$)}
with cardinality requirement.
\end{theorem}
\begin{proof}
Two alternative proofs are given in Sections \ref{SCproof1direct} and \ref{SCproof2reduce}. 
The first proof is a self-contained direct proof, 
being a refinement of the argument in \cite{Mmultexc17} for
 (the only-if part of) Theorem \ref{THmnatiffmultexch},
whereas the second makes use of (the only-if part of) Theorem \ref{THmnatiffmultexch}
through a transformation of an M$^{\natural}$-concave function 
to a valuated matroid.
\qed
\end{proof}

The stronger form (M$^{\natural}$-EXC$_{\rm ms}$)
immediately implies (M$^{\natural}$-EXC)
as a special case with $|I|=1$,
whereas (M$^{\natural}$-EXC$_{\rm ms}$)
obviously implies (M$^{\natural}$-EXC$_{\rm m}$).
Therefore, we obtain the equivalence 
of the three exchange properties
as a corollary of Theorems \ref{THmnatiffmultexch} and \ref{THmultexchmnat-str}.

\begin{corollary} \label{COmnatexch}
For a  function
$f: 2^{N} \to \RR \cup \{ -\infty \}$,
the three conditions 
{\rm (M$^{\natural}$-EXC)},
{\rm (M$^{\natural}$-EXC$_{\rm m}$)}, and
{\rm (M$^{\natural}$-EXC$_{\rm ms}$)}
are equivalent.
\end{corollary}

\section{The first proof of Theorem~\ref{THmultexchmnat-str}}
\label{SCproof1direct}

In this section we give a self-contained direct proof of Theorem~\ref{THmultexchmnat-str}.
This is a refinement of the argument in \cite{Mmultexc17} for
 (the only-if part of) Theorem \ref{THmnatiffmultexch}.

The proof is based on the Fenchel-type duality theorem 
in discrete convex analysis
(\cite[Theorem 3.1]{Mvmfen98}, \cite[Theorem 8.21 (1)]{Mdcasiam}),
which is stated below in a form convenient for our use.

\begin{theorem}[Fenchel-type duality] \label{THfenchelmnatsetfn}
Let 
$f_{1}, f_{2}: 2^{N} \to \RR \cup \{ -\infty \}$
be M$^{\natural}$-concave functions,
and  
$g_{1}, g_{2}: \RR^{N} \to \RR$
be their (convex) conjugate functions defined by 
$g_{i}(q) =  
 \max_{J \subseteq N} \{  f_{i}(J) - \sum_{j \in J} q_{j} \}$
$(i=1,2)$ for $q \in \RR^{N}$.
Then%
\footnote{
The assumption $\dom g_{1} \cap \dom g_{2} \not= \emptyset$
in \cite[Theorem 8.21 (1)]{Mdcasiam}
is satisfied, since $\dom g_{1} = \dom g_{2} = \RR^{N}$.
} 
\begin{align}
 \max_{J \subseteq N}  \{ f_{1}(J) + f_{2}(J) \}
 = \inf_{q \in \RR^{N}} \{ g_{1}(q) + g_{2}(-q) \},
\label{fencmaxmin0}
\end{align}
where the maximum on the left-hand side is
defined to be $-\infty$ 
if $\dom f_{1} \cap \dom f_{2} = \emptyset$.
If $f_{1}$ and $f_{2}$ are integer-valued, the vector $q$ can be restricted to integers. 
\end{theorem}

We also need the following
consequence of the exchange property (M$^{\natural}$-EXC).

\begin{lemma} \label{LMmexcP2P3}
If $f$ satisfies {\rm (M$^{\natural}$-EXC)},
then, for any $X, Y$ with $|X| \leq |Y|$ 
and $i \in X \setminus Y$, 
there exists $j \in Y \setminus X$ such that
$f( X) + f( Y )   \leq  f( X - i + j) + f( Y + i -j)$.
\end{lemma}
\begin{proof}
This is a direct translation of the exchange property (ii)
of  (M$^{\natural}$-EXC$_{\rm p}$) 
given in  \cite[Theorem 4.2]{MS99gp}
for M$^{\natural}$-convex function on $\mathbb{Z}\sp{N}$.
\qed
\end{proof}

We prove Theorem~\ref{THmultexchmnat-str} 
in Sections \ref{SCprooftoFenc} to \ref{SCproofevalFenc}.
In Section \ref{SCprooftoFenc}
 the stronger multiple exchange property 
(M$^{\natural}$-EXC$_{\rm ms}$)
is reformulated in terms of the conjugate functions 
by the Fenchel-type duality.
The submodularity of the conjugate functions is revealed 
in Section \ref{SCproofsubm}
and the dual objective function is evaluated
in Section \ref{SCproofevalFenc}.

\subsection{Translation to the conjugate functions}
\label{SCprooftoFenc}

Let $f: 2^{N} \to \RR \cup \{ -\infty \}$ 
be an M$^{\natural}$-concave function,
$X, Y \in \dom f$ and $I \subseteq X \setminus Y$.
To express the size constraint with bound $k$, 
we define
$\beta(J;k)=0$ if $|J| \leq k$ and 
$\beta(J;k)= -\infty$ otherwise.
We shall prove
\begin{align}
f( X) + f( Y )   \leq 
 \max_{J \subseteq Y \setminus X} 
 \{  f((X \setminus I) \cup J) +f((Y \setminus J) \cup I) + \beta(J;|I|) \} ,
\label{mnatconcavexcmult2}
\end{align}
which is equivalent to (\ref{mnatconcavexcmult-str}).
With the notations
\begin{align}
 &C = X \cap Y,
\qquad
 X_{0} = X \setminus Y = X \setminus C,
\qquad
 Y_{0} = Y \setminus X = Y \setminus C ,
\label{mexcCX0Y0def}
\\
& f_{1}(J) =  f((X \setminus I) \cup J) 
  = f( (X_{0} \setminus I) \cup C \cup J)
\qquad (J \subseteq Y_{0}),
\label{mexcf1def}
\\
& \tilde f_{1}(J) = f_{1}(J) + \beta(J;|I|) 
  = f( (X_{0} \setminus I) \cup C \cup J) + \beta(J;|I|) 
\qquad (J \subseteq Y_{0}),
\label{mexcf1defb}
\\
& f_{2}(J) = f((Y \setminus J) \cup I)
  = f(  I \cup C \cup (Y_{0} \setminus J) )
\qquad (J \subseteq Y_{0}),
\label{mexcf2def}
\end{align}
the inequality (\ref{mnatconcavexcmult2}) is rewritten as 
\begin{align}
f( X) + f( Y )   \leq 
 \max_{J \subseteq Y_{0}}  \{ \tilde f_{1}(J) + f_{2}(J)  \}.
\label{mnatconcavexcmult3}
\end{align}

\begin{lemma} \label{LMf1f2dom}
{\rm (1)}
$\dom f_{1}$, $\dom \tilde f_{1}$, and $\dom f_{2}$ are
nonempty.
\\
{\rm (2)}
$f_{1}$, $\tilde f_{1}$, and $f_{2}$ 
are M$^{\natural}$-concave functions.
\end{lemma}
\begin{proof}
(1)
We prove  
$\dom \tilde f_{1} \not= \emptyset$ and $\dom f_{1} \not= \emptyset$
by showing
\begin{align}
\mbox{there exists 
$J \subseteq Y \setminus X$ such that
$(X \setminus I) \cup J \in \dom f$ 
and $|J| \leq |I|$}
\label{f1domnonempty}
\end{align}
by induction on $|I|$.
If $|I| = 0$, (\ref{f1domnonempty}) holds trivially with $J=\emptyset$. 
Suppose $|I| \geq 1$ and $I = I' + i$ with $i \not\in I'$.
By the induction hypothesis there exists
 $J' \subseteq Y \setminus X$ such that
$X' := (X \setminus I') \cup J' \in \dom f$
and $|J'| \leq |I'|$. 
By (M$^{\natural}$-EXC) for $(X', Y, i)$, 
(i) $X' - i \in \dom f$ or
(ii) there exists $j \in Y \setminus X'\  (\subseteq Y \setminus X)$ such that
$X' - i + j \in \dom f$.
In case (i), we set $J = J'$ to obtain
$|J| = |J'| \leq |I'| < |I|$
and
$(X \setminus I) \cup J = (X \setminus (I' + i)) \cup J' = X' -  i \in \dom f$.
In case (ii), we set $J = J' + j$ to obtain
$|J| = |J'|+1 \leq |I'|+1 =|I|$
and
$(X \setminus I) \cup J = ( X \setminus (I' + i) ) \cup (J' + j) = X' -  i + j \in \dom f$.
Thus  (\ref{f1domnonempty}) is proved.

To prove $\dom f_{2} \not= \emptyset$,
we show
\begin{align}
\mbox{there exists 
$J \subseteq Y \setminus X$ such that
$(Y \setminus J) \cup I \in \dom f$}
\label{f2domnonempty}
\end{align}
by induction on $|I|$ (by almost the same argument as above).
If $|I| = 0$, (\ref{f2domnonempty}) holds trivially with $J=\emptyset$. 
Suppose $|I| \geq 1$ and $I = I' + i$ with $i \not\in I'$.
By the induction hypothesis there exists
 $J' \subseteq Y \setminus X$ such that
$Y' := (Y \setminus J') \cup I' \in \dom f$.
By (M$^{\natural}$-EXC) for $(X, Y', i)$, 
(i) $Y' + i \in \dom f$ or
(ii) there exists $j \in Y' \setminus X \  (\subseteq Y \setminus X)$ such that
$Y' + i - j \in \dom f$.
In case (i), we set $J = J'$ to obtain
$(Y \setminus J) \cup I = (Y \setminus J') \cup (I' + i)  = Y' +  i \in \dom f$.
In case (ii), we set $J = J' + j$ to obtain
$(Y \setminus J) \cup I =(Y \setminus (J' + j) ) \cup (I' + i)  = Y' +  i - j \in \dom f$.
Thus  (\ref{f2domnonempty}) is proved.

(2)
For $f_{1}$ and $f_{2}$,
the M$^{\natural}$-concavity is easy to see from 
(M$^{\natural}$-EXC) of $f$.
Then the function $\tilde f_{1}$, being a restriction of $f_{1}$, 
is also M$^{\natural}$-concave.
\qed
\end{proof}

Consider the (convex) conjugate functions of $\tilde f_{1}$ and $f_{2}$
given by
\begin{align}
\tilde g_{1}(q) &=  
 \max_{J \subseteq Y_{0}} \{ \tilde f_{1}(J) - q(J) \} 
\qquad (q \in \RR^{Y_{0}}),
\label{tg1def}
\\
g_{2}(q) &= 
 \max_{J \subseteq Y_{0}} \{  f_{2}(J) - q(J) \} 
\qquad (q \in \RR^{Y_{0}}),
\label{g2def}
\end{align}
where $q(J) = \sum_{j \in J} q_{j}$.
By Theorem~\ref{THfenchelmnatsetfn}
the desired inequality (\ref{mnatconcavexcmult3})
can be rewritten as
\begin{align}
f( X) + f( Y )   \leq 
 \inf_{q \in \RR^{N}} \{ \tilde g_{1}(q) + g_{2}(-q) \}.
\label{mnatconcavexcmult4}
\end{align}

\subsection{Submodularity}
\label{SCproofsubm}

To compute $\tilde g_{1}(q) + g_{2}(-q)$ 
in (\ref{mnatconcavexcmult4})
we relate $\tilde g_{1}$ and $g_{2}$, respectively, 
to
\begin{align}
\tilde g(p) &=  \max_{Z \subseteq N} \{  f(Z) + \beta(Z; |X|) - p(Z) \} 
\qquad (p \in \RR^{N}),
\label{gpdef-size}
\\
g(p) &=  \max_{Z \subseteq N} \{  f(Z) - p(Z) \} 
\qquad (p \in \RR^{N}).
\label{gpdef}
\end{align}
We use notation
$f[-p](Z) = f(Z) - p(Z)$ for $Z \subseteq N$.

Since 
$f(Z) + \beta(Z; |X|)$ and $f(Z)$
are M$^{\natural}$-concave,
the conjugacy theorem in discrete convex analysis
(\cite[Theorems 8.4, (8.10)]{Mdcasiam}, \cite[Theorem 3.4]{Mbonn09})
 shows that
both $\tilde g$ and $g$ are L$^{\natural}$-convex functions on $\RR\sp{N}$.
In particular, they are submodular:
\begin{align} 
  \tilde g(p) + \tilde g(p') & \geq \tilde g(p \vee p') + \tilde g(p \wedge p')
\qquad (p, p' \in \RR^{N}) ,
\label{tgsubm}
\\
  g(p) +  g(p') & \geq  g(p \vee p') + g(p \wedge p')
\qquad (p, p' \in \RR^{N}) ,
\label{gsubm}
\end{align}
where $p \vee p'$ and $p \wedge p'$ denote, respectively, the vectors of component-wise
maximum and minimum of $p$ and $q$.

For our proof we need the following form of submodularity across
$\tilde g$ and $g$.

\begin{lemma} \label{LMsbmtgg}
\quad
$  \tilde g(p) + g(q) \geq  \tilde g(p \wedge q) + g(p \vee q) $ 
$\qquad (p, q \in \RR^{N})$.
\end{lemma}
\begin{proof}
It follows from Lemma~\ref{LMtggstrongquot} below and (\ref{gsubm}) that
\begin{align*}
 \tilde g(p) - \tilde g(p \wedge q) 
\geq  g(p) - g(p \wedge q) 
\geq  g(p \vee q) - g(q) ,
\end{align*}
which is equivalent to the claim.
\qed
\end{proof}

\begin{lemma} \label{LMtggstrongquot}
For any $p,q \in \RR\sp{N}$ with $p \geq q$, it holds%
\footnote{
(\ref{tggstrquot}) means a kind of strong quotient relation.
}
\begin{align}
& \tilde g(p) - \tilde g(q) \geq  g(p) - g(q) .
\label{tggstrquot}
\end{align}
\end{lemma}
\begin{proof}
The assertion (\ref{tggstrquot}) is equivalent to the monotonicity
of $\tilde g(p) - g(p)$ in $p$.
To prove this it suffices to show
that for each $q \in \RR^{N}$
there exists a positive number $\varepsilon(q) > 0$
such that 
(\ref{tggstrquot}) holds
for all $p \in \RR^{N}$ of the form
\begin{equation} \label{p=q+alphak}
p = q + \alpha \chi_{k}
\end{equation}
with $0 \leq \alpha < \varepsilon(q)$,
where $\chi_{k}$ denotes the $k$th unit vector for $k \in N$. 
We will show that
the minimum of the nonzero absolute values of 
$f[-q](Z_{1})+ f[-q](Z_{2})-f[-q](Z_{3}) -f[-q](Z_{4})$
over all $Z_{1}, Z_{2}, Z_{3}, Z_{4} \subseteq N$
serves as such $\varepsilon(q)$.
We define
\begin{align*}
\varepsilon(q) =  \min\{  |f[-q](Z_{1}) \! + \!  f[-q](Z_{2}) \! - \! 
            f[-q](Z_{3}) \! - \! f[-q](Z_{4})| \not= 0 
  \mid  Z_{1}, Z_{2}, Z_{3}, Z_{4} \subseteq N \}.
\end{align*}

Recalling 
(\ref{gpdef-size}) and (\ref{gpdef}), 
denote $m=|X|$ and take $U$ and $W$ 
such that
\[
 g(p) = f(U) - p(U),
\quad 
\tilde g(q) = f(W) - q(W),
\quad 
|W| \leq m.
\]
We choose such $U$, $W$ with minimum $|W \setminus U|$.
Then (\ref{tggstrquot})
 is rewritten as 
\begin{align}
 [f(U) - p(U)] + [f(W) - q(W)] 
 \leq  \tilde g(p) + g(q). 
\label{tggstrquot2}
\end{align}
This inequality can be shown as follows.
\begin{itemize}
\item
If $|U| \leq m$, we have
$f(U) - p(U) \leq  \tilde g(p)$
by (\ref{gpdef-size}) 
as well as 
$f(W) - q(W) \leq  g(q)$
by (\ref{gpdef}).
Hence (\ref{tggstrquot2}) holds.

\item
If $W \subseteq U$, then  
$p(U) + q(W)  \geq p(W) + q(U)$ by $p \geq q$, and hence
\[
 [f(U) - p(U)] + [f(W) - q(W)] 
  \leq 
 [ f(W)  - p(W) ] + [ f(U) - q(U) ]
\leq  \tilde g(p) + g(q), 
\]
which shows (\ref{tggstrquot2}).

\item
The remaining case,
where $|U| > m$ and $W \setminus U \not= \emptyset$,
is excluded by the minimality of $|W \setminus U|$,
as shown below.
\end{itemize}

Suppose that
$|U| > m$ and $W \setminus U \not= \emptyset$.
Then $|U| > m \geq |W|$.
Take any $i \in W \setminus U$, which is possible since
$W \setminus U \not= \emptyset$.
By Lemma~\ref{LMmexcP2P3} 
there exists $j \in U \setminus W$ such that 
\begin{align}
 f(W) + f(U)  \leq  f( W - i + j) + f( U + i -j) =  f(W') + f(U') ,
\label{fWfUfW'fU'}
\end{align}
where $W' = W - i + j$ and $U'= U + i -j$.
Note that $|W'| =|W| \leq m$.

\begin{itemize}
\item
Case of $k \not\in W \setminus U$:
Since $i \not= k$, we have $p_{i} = q_{i}$ and $p_{j} \geq q_{j}$.
Then, by (\ref{fWfUfW'fU'}), we have
\begin{align*}
 & [f(U) - p(U)] + [f(W) - q(W)] 
\\
 &\leq  
 [f( U + i -j) - p( U + i -j)]
 + [f( W - i + j) - q( W - i + j)] 
\\
 &= [f(U') - p(U')] + [f(W') -q(W')] .
\end{align*}
Since $|W'| =|W| \leq m$, this means 
$ f(W')-q(W') =  f(W)-q(W)$ as well as
$f(U') - p(U') = f(U) - p(U)$,
whereas
$W' \setminus U = (W \setminus U) - i$.
This is a contradiction to the minimality of 
$|W \setminus U|$.

\item
Case of $k \in W \setminus U$:
We choose $i=k$ in (\ref{fWfUfW'fU'}) and rewrite (\ref{fWfUfW'fU'}) as
\begin{align}
 f[-q](W) + f[-q](U)  \leq   f[-q](W') + f[-q](U') .
\label{fWfUfW'fU'2}
\end{align}
Here we have
\begin{align}
& f[-q](W') \leq f[-q](W),
\label{fW'fW}
\\  
& f[-q](U') - \alpha = f[-p](U')  \leq f[-p](U) = f[-q](U) 
\label{fU'fU}
\end{align}
by the definitions of $W$ and $U$,  (\ref{p=q+alphak}), $k \in U'$, and $k \not\in U$.
Hence the difference of both sides
of (\ref{fWfUfW'fU'2}) is at most $\alpha$,
whereas $\alpha < \varepsilon(q)$.
Hence we have equality in (\ref{fWfUfW'fU'2}),
and therefore 
$f[-q](W') = f[-q](W)$ in (\ref{fW'fW}).
This is a contradiction to the minimality of 
$|W \setminus U|$, since 
$|W' \setminus U| < |W \setminus U|$.
\end{itemize}
\qed
\end{proof}

\subsection{Evaluation of the Fenchel dual}
\label{SCproofevalFenc}

The desired inequality (\ref{mnatconcavexcmult4})
follows from the following lemma, whose proof uses
Lemma~\ref{LMsbmtgg}.

\begin{lemma} \label{LMg1qg2q}
For any $q \in \RR^{Y_{0}}$, we have
$\tilde g_{1}(q) + g_{2}(-q) \geq f( X) + f( Y )$.
\end{lemma}

\begin{proof}
For a vector $q \in \RR^{Y_{0}}$ we define 
$p^{(1)}, p^{(2)} \in \RR^{N}$ by
\begin{align*} 
p^{(1)}_{i}  &= p^{(2)}_{i}  =
   \left\{  \begin{array}{ll}
    q_{i}          &   (i  \in Y_{0}) ,     \\
   - M     &   (i \in C ),  \\
   + M     &   (i \in N \setminus (X \cup Y)  ) , \\
                     \end{array}  \right.
\quad 
p^{(1)}_{i}  = - p^{(2)}_{i}  =
   \left\{  \begin{array}{ll}
   - M     &   (i \in X_{0} \setminus I ),  \\
   + M     &   (i \in I ),  \\
                     \end{array}  \right.
\end{align*}
where $M$ is a sufficiently large positive number.

The maximizer $Z$ 
of $\tilde g(p)$ in (\ref{gpdef-size}) 
for $p = p^{(1)}$ 
must avoid $I$ and include $(X_{0} \setminus I) \cup C$. 
Hence
$Z =(X_{0} \setminus I) \cup C \cup J$ 
for some $J \subseteq Y_{0}$,
and then
\begin{align*}
& |Z| \leq |X| \iff |J| \leq |I|,
\\ &
p^{(1)}(Z) = - M (|X_{0} \setminus I|+|C|) + q(J).
\end{align*}
Therefore, we have
\begin{align}
\tilde g(p^{(1)}) &=
 \max_{Z \subseteq N} \{ f( Z ) + \beta(Z; |X|)- p^{(1)}(Z) \} 
\notag \\ & =
 \max_{J \subseteq Y_{0}} \{ f( (X_{0} \setminus I) \cup C \cup J) 
+ \beta(J; |I|)- q(J) \} 
 + M (|X_{0} \setminus I|+|C|)
\notag \\ & =
\tilde g_{1}(q) 
 + M (|X_{0} \setminus I|+|C|).
\label{tgp1}
\end{align}
The maximizer $Z$ of 
$g(p)$ in (\ref{gpdef}) 
for $p = p^{(2)}$
must include $I \cup C$ and avoid $X \setminus (I \cup C)$.
Hence
$Z = I \cup C \cup (Y_{0} \setminus J)$ 
for some $J \subseteq Y_{0}$,
and then
\begin{align*}
p^{(2)}(Z) = - M (|I|+|C|) + q(Y_{0} \setminus J).
\end{align*}
Therefore, we have
\begin{align}
 g(p^{(2)}) &=
 \max_{Z \subseteq N} \{ f( Z )  - p^{(2)}(Z) \} 
\notag \\ & =
 \max_{J \subseteq Y_{0}} \{ f( I \cup C \cup (Y_{0} \setminus J) ) + q(J) \} 
 - q(Y_{0}) + M (|I|+|C|)
\notag \\ & =
 g_{2}(-q) 
 - q(Y_{0}) + M (|I|+|C|).
\label{gp2}
\end{align}
By adding (\ref{tgp1}) and (\ref{gp2}) we obtain
\begin{align}
 \tilde g_{1}(q) + g_{2}(-q)
 & =
 \tilde g(p^{(1)}) + g(p^{(2)}) - M (|X|+|C|) + q(Y_{0}).
\label{tg1g2gg}
\end{align}
By Lemma~\ref{LMsbmtgg} we have
\begin{align}
 \tilde g(p^{(1)}) + g(p^{(2)}) \geq 
 \tilde g(p^{(1)} \wedge p^{(2)})
 + g(p^{(1)} \vee p^{(2)}) .
\label{tggsbmp1p2}
\end{align}
Since
\begin{align*} 
& (p^{(1)} \vee p^{(2)})_{i}   = (p^{(1)} \wedge p^{(2)})_{i}  =
   \left\{  \begin{array}{ll}
    q_{i}          &   (i  \in Y_{0}) ,     \\
   - M     &   (i \in C ),  \\
   + M     &   (i \in N \setminus (X \cup Y)  ) , \\
                     \end{array}  \right.
\\
& (p^{(1)} \vee p^{(2)})_{i}   = - (p^{(1)} \wedge p^{(2)})_{i}  =
   + M    \quad   (i \in X_{0}),  \\
\end{align*}
we have
\begin{align} 
\tilde g(p^{(1)} \wedge p^{(2)}) & \geq
f(X) + M |X|,
\label{gpwedgep}
\\
 g(p^{(1)} \vee p^{(2)}) & \geq
f(Y) - q(Y_{0}) + M|C|,
\label{gpveep}
\end{align}
where
(\ref{gpwedgep}) follows from (\ref{gpdef-size}) with $Z=X$
and
(\ref{gpveep}) follows from  (\ref{gpdef}) with $Z=Y$.
The combination of  
(\ref{tg1g2gg}),
(\ref{tggsbmp1p2}),
(\ref{gpwedgep}), and
(\ref{gpveep})
yields the desired inequality
$\tilde g_{1}(q) + g_{2}(-q) \geq f(X) + f(Y)$.
\qed
\end{proof}

We have thus completed the proof of Theorem~\ref{THmultexchmnat-str}.

\begin{remark} \label{RMintvalued} \rm
For an integer-valued function $f: 2^{N} \to \ZZ \cup \{ -\infty \}$,
the above proof  can be made purely discrete.
In particular,
the integrality in the Fenchel-type duality in Theorem \ref{THfenchelmnatsetfn}
allows us to assume  $p$ and $q$ to be integer vectors.
In the proof of Lemma~\ref{LMtggstrongquot} 
we assume $p = q +  \chi_{k}$,
with $\alpha = 1$ in (\ref{p=q+alphak}).
At the end of the proof of Lemma~\ref{LMtggstrongquot}, in the case where
$|U| > m$ and $k \in W \setminus U$,
the inequalities
 (\ref{fWfUfW'fU'2}), (\ref{fW'fW}), and (\ref{fU'fU})
together with integrality
yield at least one of the following:
(i)
$f[-q](W') = f[-q](W)$ and
(ii)
$f[-p](U') = f[-p](U)$.
This is a contradiction to the minimality of 
$|W \setminus U|$, since 
in case (i) we can replace $W$ to $W'$ to obtain
$|W' \setminus U| < |W \setminus U|$,
and 
in case (ii) we can replace $U$ to $U'$ to obtain
$|W\setminus U'| < |W \setminus U|$.
\finbox
\end{remark}

\section{The second proof of Theorem~\ref{THmultexchmnat-str}}
\label{SCproof2reduce}

The second proof transforms a given M$^{\natural}$-concave function $f$ 
to an M-concave function (valuated matroid) $\hat{f}$, 
and then applies the only-if part of
 Theorem \ref{THmnatiffmultexch} to $\hat{f}$
in its special case for M-concave functions.

A function $f: 2^{N} \to \RR \cup \{ -\infty \}$ 
with $\dom f \not= \emptyset$
is called an M-concave function \cite{Mdcasiam}
(valuated matroid \cite{DW90,DW92})
if,
for any $X, Y \subseteq N$ and $i \in X \setminus Y$,
it holds that
\begin{align}
f( X) + f( Y )   \leq 
 \max_{j \in Y \setminus X}  \{ f( X - i + j) + f( Y + i -j) \} .
\label{valmatexc}
\end{align}
We can also say that an M-concave function is nothing but 
an M$^{\natural}$-concave function $f$ such that
$\dom f$ consists of equi-cardinal subsets, i.e.,
$|X| = |Y|$ for any $X, Y \in \dom f$.
Therefore,  Theorem~\ref{THmnatiffmultexch} 
in this special case shows that
every M-concave function has the multiple exchange property 
{\rm (M$^{\natural}$-EXC$_{\rm m}$)} 
with the additional condition $|J|=|I|$.

Let $f: 2^{N} \to \RR \cup \{ -\infty \}$
be an  M$^{\natural}$-concave function.
Denote by $r$ and $s$ the maximum and minimum, respectively, of $|X|$ for $X \in \dom f$,
and define
$S = \{ n+1,n+2,\ldots, n+(r-s) \}$
and
$\hat{N} = N \cup S = \{ 1,2,\ldots, \hat n \}$,
where $\hat n =n+(r-s)$.
Define
$\hat{f}: 2^{\hat N} \to \RR \cup \{ -\infty \}$ by
\begin{align} \label{assocMdef} 
\hat{f}(Z)  =
   \left\{  \begin{array}{ll}
    f(Z \cap N)         &   (|Z| = r) ,     \\
   -\infty    &   (\mbox{otherwise}) .  \\
                     \end{array}  \right.
\end{align}
That is, for $X \subseteq N$ and $U \subseteq S$,
we have $\hat{f}(X \cup U) = f(X)$ 
if $|U|=r - |X|$.
By Lemma~\ref{LMassocMfn} below,
$\hat{f}$ is an M-concave function.

Suppose that we are given $X, Y \in \dom f$ and a subset 
$I \subseteq X \setminus Y$.
Take any $U, W \subseteq S$ with
$|U|=r - |X|$ and $|W|=r - |Y|$.
Then
$X \cup U,  Y \cup W  \in \dom \hat f$
and
$I \subseteq (X \cup U) \setminus (Y \cup W)$.
By Theorem \ref{THmnatiffmultexch} for $\hat f$, 
there exists 
$J \subseteq Y \setminus X$
and $V \subseteq W \setminus U$  such that
\begin{align*}
& 
\hat f( X \cup U) +\hat  f( Y \cup W )  
\\ & \leq  
  \hat f( \, ( (X \setminus I) \cup J )  \cup (U  \cup V )  \,  ) +
\hat f( \,  ( (Y \setminus J) \cup I) \cup (W  \setminus V ) \,  ),
\end{align*}
which implies
$f( X) + f( Y )   \leq  f((X \setminus I) \cup J) +f((Y \setminus J) \cup I)$.
Since $\dom \hat f$ consists of equi-cardinal sets,
we must have
$|I| = |J| + |V|$,
which shows
$|I| \geq |J|$.

\begin{lemma} \label{LMassocMfn}
For an M$^{\natural}$-concave function $f$, the function $\hat{f}$ 
in {\rm (\ref{assocMdef})} is M-concave.
\end{lemma}
\begin{proof}
Let $X, Y \in \dom f$ and  $U, W \subseteq S$ with
$|U|=r - |X|$ and $|W|=r - |Y|$.
The exchange property for $\hat{f}$ amounts to the following: 
\begin{itemize}
\item
For any $i \in X \setminus Y$
there exists  
$j \in Y \setminus X$ with (\ref{assocM11})
or
$j \in W \setminus U$ with (\ref{assocM12}),
where
\begin{align}
&  \hat f( X \cup U) +\hat  f( Y \cup W )  
\leq   \hat f( (X - i + j )  \cup U  ) + \hat f( (Y + i -j) \cup W  ),
\label{assocM11}
\\
&  \hat f( X \cup U) +\hat  f( Y \cup W )  
\leq   \hat f( (X - i)  \cup (U +j)  ) + \hat f( (Y + i) \cup (W-j)  ).
\label{assocM12}
\end{align}

\item
For any $i \in U \setminus W$
there exists  
$j \in Y \setminus X$ with (\ref{assocM21})
or
$j \in W \setminus U$ with (\ref{assocM22}),
where
\begin{align}
&  \hat f( X \cup U) +\hat  f( Y \cup W )  
\leq   \hat f( (X + j )  \cup (U -i)  ) + \hat f( (Y -j) \cup (W+i)  ),
\label{assocM21}
\\
&  \hat f( X \cup U) +\hat  f( Y \cup W )  
\leq   \hat f( X  \cup (U -i +j)  ) + \hat f( Y  \cup (W+i-j)  ).
\label{assocM22}
\end{align}
\end{itemize}
The exchange properties above can be shown as follows.
For any $i \in X \setminus Y$.
we have (\ref{mconcav1}) or (\ref{mconcav2}).
In case of (\ref{mconcav2}) we obtain (\ref{assocM11}).
In case of (\ref{mconcav1}) we obtain (\ref{assocM12})
for any $j \in W \setminus U$, if $W \setminus U$ is nonempty.
If $W \setminus U$ is empty, then $|X| \leq |Y|$ and
we have (\ref{assocM11}) by Lemma \ref{LMmexcP2P3}.
Next, take any $i \in U \setminus W$.
If $W \setminus U$ is nonempty, (\ref{assocM22}) holds 
for any $j \in W \setminus U$.
If $W \setminus U$ is empty, we have
$|U| > |W|$ and hence $|X| < |Y|$.
Then Lemma \ref{LMmexcP1} below shows (\ref{assocM21}).
\qed
\end{proof}

\begin{lemma} \label{LMmexcP1}
If $f$ satisfies {\rm (M$^{\natural}$-EXC)},
then, for any $X, Y$ with $|X| < |Y|$,
there exists $j \in Y \setminus X$ such that
$f( X) + f( Y )   \leq  f( X  + j) + f( Y -j)$.
\end{lemma}
\begin{proof}
This is a direct translation of the exchange property (i)
of  (M$^{\natural}$-EXC$_{\rm p}$) 
given in  \cite[Theorem 4.2]{MS99gp}
for M$^{\natural}$-convex function on $\mathbb{Z}\sp{N}$.
\qed
\end{proof}

\section*{Acknowledgments}
The author thanks Akiyoshi Shioura
for suggesting a simplification in the proof
of Section \ref{SCproof1direct}.
He is also thankful to Kenjiro Takazawa and Akihisa Tamura for helpful comments.
This work is supported by The Mitsubishi Foundation, 
CREST, JST, Grant Number JPMJCR14D2, Japan, and 
KAKENHI Grant Number 26280004.

\end{document}